\documentclass{article}

\usepackage{amsmath}
\usepackage{amsthm}
\usepackage{amssymb}
\usepackage{color}
\usepackage{enumerate}
\usepackage[margin=2cm]{geometry}
\usepackage{graphicx}
\usepackage{mathrsfs}
\usepackage[latin1]{inputenc}
\usepackage{subfigure,xspace}
\usepackage[colorinlistoftodos]{todonotes}
\usepackage{comment}
\usepackage{ulem}
  
\newtheorem{theorem}{Theorem}[section]
\newtheorem{definition}{Definition}
\newtheorem{assumption}{Assumption}
\newtheorem{lemma}{Lemma}
\newtheorem{proposition}{Proposition}

\usepackage{amsmath,mathrsfs}
\usepackage{amssymb}

\newcommand{\myclearpage}{\clearpage}
\renewcommand{\myclearpage}{}

\def\one{\mathbf{1}}

\def\ps{p^{*}}

\def\R{\mathbb{R}}
\newcommand{\Rn}{\mathbb{R}^n}
\def\Rn{\mathbb{R}^n}
\def\S{\mathbb{S}}

\def\u{\mathbf{u}}

\def\w{\mathbf{w}}
\def\x{\mathbf{x}}
\def\y{\mathbf{y}}
\def\z{\mathbf{z}}

\def\V{\mathsf{V}}

\def\zero{\mathbf{0}}
\newcommand{\ones}{\mathbf{1}}

\newcommand{\argmin}{\operatornamewithlimits{argmin}}

\newcommand{\fLift}{\tilde{F}}

\newcommand{\Lip}{K}

\newcommand{\G}{\mathcal{G}}
\def\E{\mathcal{E}}

\newcommand{\map}[3]{#1:#2 \rightarrow #3}

\newcommand{\real}{{\mathbb{R}}}

\newcommand{\integers}{\mathbb{Z}}

\newcommand{\gradient}{\nabla}

\title{A Unified Framework for Continuous-time Unconstrained Distributed Optimization}

\author{Behrouz Touri and Bahman Gharesifard \thanks{
Behrouz Touri is with the Department of Electrical and Computer at the University of California San Diego,~\texttt{email: btouri@ucsd.edu} and 
Bahman Gharesifard is with
the Department of Electrical and Computer Engineering at the University of California Los Angeles, ~\texttt{email: gharesifard@ucla.edu}.}}

\begin{document}

\maketitle

\begin{abstract}
We introduce a class of distributed nonlinear control systems, termed as the flow-tracker dynamics, which capture phenomena where the average state is controlled by the average control input, with no individual agent has direct access to this average. The agents update their estimates of the average through a nonlinear observer. We prove that utilizing a proper gradient feedback for any distributed control system that satisfies these conditions will lead to a solution of the corresponding distributed optimization problem. We show that many of the existing algorithms for solving distributed optimization are instances of this dynamics and hence, their convergence properties can follow from its properties. In this sense, the proposed method establishes a unified framework for distributed optimization in continuous-time. Moreover, this formulation allows us to introduce a suit of new continuous-time distributed optimization algorithms by readily extending the graph-theoretic conditions under which such dynamics are convergent. 
\end{abstract}

\myclearpage

\section{Introduction}\label{sec:intro}

Many scenarios of cooperative coordination of multi-agent systems can be cast as agreement-based distributed control problems, where the decisions of the individual agents are driven by a combination of estimates obtained by averaging the states of their neighbouring agents and external control inputs, computable using local information. An important subclass of such problems which is the subject of this paper is the class of {distributed optimization problem or distributed learning} problems, where agents have access to a private function and their objective is to find a  minimizer of the sum of these function using local information. The importance of the problem stems from its applications, including distributed electricity generation and smart grids~\cite{BAR-CNH-ADD:13,ADD-STC-CNH:12,NL:13} and sensor networks~\cite{MR-RN:04}. Many coordination algorithms for distributed optimization are constructed by a combination of an information aggregation and a gradient flow of agent's individual functions~\cite{JNT-DPB-MA:86, MR-RN:04,LX-SB:06, AN-AO:09,PW-MDL:09,AN-AO-PAP:10,BJ-MR-MJ:09,MZ-SM:12, JW-NE:10,JW-NE:11, BG-JC:14-tac,burger2014polyhedral,mateos2015distributed,ren2005consensus, aybat2015asynchronous, zeng2017fast, li2019decentralized, alghunaim2020decentralized}. The external control input is hence given by the gradient of individual objective functions. We refer to this subclass the \emph{consensus-based  distributed optimization algorithms}. It is worth mentioning that there are other classes of distributed optimization protocols, for example, when the underlying functions are separable but there are local coupling constraints~\cite{niederlander2016distributed}, or when the communication constraints are cast as linear constraints and primal-dual methods are employed~\cite{wei2013distributed}. 


Distributed consensus-based convex optimization algorithms are either developed in discrete-time or in continuous-time. Although the main focus of this work is on developing continuous-time dynamics, in what follows next, we briefly review the literature on both classes of algorithms, admittedly missing some relevant references. 

\emph{Discrete-time dynamics:} Most of the available consensus-based algorithms on distributed optimization are in discrete-time. These dynamics are commonly first order, in the sense that the agents need to only carry one state, and before the work~\cite{Nedic2013}, contained the restrictive assumption that the network topology 
is either undirected, or directed but doubly stochastic, see~\cite{AN-AO:09, PW-MDL:09,AN-AO-PAP:10,BJ-MR-MJ:09,MZ-SM:12}. Another common feature in many of the discrete-time dynamics is that the stepsize is time-varying. Similar to the average consensus-dynamics~\cite{dominguez2013distributed}, the so-called push-sum protocol~\cite{kempe2003gossip} can be utilized to overcome the restrictive doubly stochastic assumption in distributed optimization. In fact, using a perturbed push-sum protocol, the work~\cite{Nedic2013}, and a large volume of references thereafter which we are unable to adequately review but point out particularly to~\cite{AN-AO-WS-17} for geometric rates and to~\cite{PR-BG-TL-BT:18-ifac,PR-BG-TL-BT:20} for random settings, provide a subgradient-push distributed optimization protocol, provably convergent to a minimizer of the sum of convex functions on any uniformly strongly connected sequence of time-varying directed graphs.  As we will describe shortly, one of the consequences of our work is the extension of this technique to continuous-time.

\emph{Continuous-time dynamics:} Unlike the discrete-time algorithms, the literature on continuous-time consensus-based distributed optimization is not large. The first continuous-time strategy for distributed optimization is introduced in~\cite{JW-NE:10,JW-NE:11}, where the graph topology is assumed to be undirected. The formal analysis of the convergence of this result and its extensions to the sum of locally Lipschitz convex functions are provided in~\cite{BG-JC:14-tac}, and various extensions are given since~\cite{mateos2016noise,kia2015distributed}. The key property of these distributed optimization dynamics is that they are variants of saddle-point dynamical systems, which renders them second-order, and that they do not rely on time-varying stepsize. For the differentiable scenarios, this class of dynamics can be extended to handle the weight-balanced scenarios~\cite{BG-JC:14-tac}, nevertheless, such extensions are no longer saddle-point dynamics. In a our previous work~\cite{BT-BG:19-tac}, we showed that saddle-point like dynamical systems can be used alongside with a push-sum dynamics to provide continuous-time dynamics that are guaranteed to asymptotically converge to a set, where the first component is the set of optimal points, without relying on the weight-balanced assumption. Even though these results allow for time-varying network settings, a key ingredient of the proof in~\cite{BT-BG:19-tac} is the assumption that the sequence of Laplacian matrices admit a common stationary distribution. Note that such a condition is much stronger than that of the state-of the art results on  continuous-time consensus dynamics, where there has been substantial advancements on the conditions under which these dynamics are convergent, particularly~\cite{hendrickx2013convergence,martin2013continuous,Bolouki2014}, where conditions such as cut-balanced and the infinite-flow graph are proved to be enough. This leaves a major gap within the literature on continuous-time dynamics for distributed optimization when compared to discrete-time dynamics.

Our main objective in writing this manuscript 
is to provide a general separation-type results for distributed optimization and computation in the sense that any distributed dynamics that poses a herd tracking behavior (as will be defined later) and mixing of information, can be turned into a distributed optimization solver. This not only provides an overarching framework for many of the existing algorithms but also provides a behavioral approach to distributed optimization. We believe this to be important, particularly, in light of a volume of recent work on distributed optimization and advancements on convergence properties of consensus-dynamics.

\subsection{Statement of contributions} 
The contributions of this paper are the followings. We provide a unified framework for distributed optimization algorithms in continuous-time by introducing a class of distributed nonlinear control systems, which we refer to them as \textit{the flow-tracker dynamics}. Each flow-tracker dynamics is associated to a sequence of directed graphs and is assumed to be distributed over the directed graph at each time, with additional two key properties: firstly, each agent is equipped with an observer that tracks the average partial-state of the agents; secondly, the average dynamics of all agents tracks the average of the external control inputs.  As our first contribution, we consider an unconstrained distributed convex optimization problem and prove that when the external inputs are given by the gradients of the individual objective functions, scaled according to an appropriate step-size, any flow-tracker dynamics is asymptotically convergent; in particular, the state of agents' observers reach consensus to an optimizer of the sum of individual objective functions. The rest of the paper investigates the implications of this result. We demonstrate that  many known consensus-based distributed optimization protocols are flow-tracker dynamics. More importantly, we construct flow-tracker dynamics that greatly extend the class of time-varying directed graphs on which distributed optimization algorithms is provably convergent to an optimizer. In particular, we demonstrate that a continuous-time version of the perturbed push-sum protocols for distributed optimization is a flow-tracker dynamics on any sequence of time-varying directed graphs which results in a so-called class $ P^* $ weakly exponentially ergodic flow. Finally, an upshot of our results is that any distributed-averaging/consensus dynamics with exponential rate of convergence can be immediately turned into an algorithm for the distributed optimization problem.

\subsection{Organization}
The organization of this paper is as follows: in Section~\ref{section:prelim}, we introduce the mathematical notations and preliminary definitions that will be used throughout this work.  In Section~\ref{section:problem-formulation}, we formally introduce the distributed optimization problem under study. We introduce \textit{flow tracker} dynamics in Section~\ref{section:system-approach} and there, present the main result of this work. The implications of this result will be discussed in Section~\ref{section:implications}. {To assist with the flow of the paper, we postpone the presentation of the technical details and the proofs of the results to Section~\ref{section:proofs}}. Finally concluding remarks and future directions will be discussed in Section~\ref{sec:conclusions}.

\myclearpage

\section{Mathematical Preliminaries}\label{section:prelim}

We introduce some of the mathematical preliminaries and notations that we will be using throughout this paper. Let $ n $ be a positive integer. We denote the all-one and all-zero vectors in $\Rn$ by $\one$ and $\zero$, respectively. We use the shorthand notation $[n]:=\{1,\ldots,n\}$.  {The main goal of this paper is to study distributed dynamics to solve an optimization problem in $\R^d$ for some $d\geq 1$. For convenience, unless mentioned otherwise, we view all vectors in $\R^d=\R^{1\times d}$ as row-vectors.} For $X\subset \real^d $, we use bold and {lower case} letters such as $\x$ to denote a matrix $\x\in X^n\subset \R^{n\times d}$ and we use $x_i\in X$ to denote the $i$th row of $\x$. For $\x\in X^n$, we let $\bar{x} \in \real^d $ to be the average of the rows of $\x$, i.e., 
$
\bar{x}=\tfrac{1}{n}\sum_{i=1}^n x_i.
$
{Note that in this case, we view $x_i$ and $\bar{\x}$ as row vectors in $\R^d$. }
For vectors $x,y\in\R^d$, we use $x\geq y$ if $x_i\geq y_i$ for all $i\in [d]$. 

We denote the identity matrix by $ I $. We denote the $ (i,j) $-entry of a matrix $ A \in \real^{n\times n} $ by $ A_{ij} $. We denote the complement of a subset $ S \subset [n] $ by $ \bar{S} $. Let $(\V,\|\cdot\|)$ be a normed vector space. For $x\in \V$ and $A\subset \V$, we let $d(x,A)=\inf_{y\in A}\|x-y\|$. We let $\S$ to be the set of $n\times n$ column-stochastic matrices and $\S_1$ to be the subset of rank-one column stochastic matrices, i.e., the set of matrices with identical (stochastic) columns. We refer to a matrix with non-positive off-diagonal entries and columns adding to zero as a \emph{generalized Laplacian matrix}, or simply as a \emph{Laplacian matrix}. If both $L$ and $L^T$ are generalized Laplacian matrices, we say that $L$ is \emph{weight-balanced}. We denote the $i$th row of an $n\times n$ matrix $ L\in \real^{n\times n} $ by $L_i$ and we denote the $j$th column by $L^j$.

For two non-empty proper subsets $S_1,S_2\subset [n]$, we let
\begin{align}\label{eqn:LSS}
L_{S_1S_2}=\sum_{i\in S_1}\sum_{j\in S_2}L_{ij}.
\end{align}
{Throughout this paper, we use the $\ell_2$-norm $\|\cdot\|$ for vectors in $\R^n$ and the resulting induced norm for matrices, which for the simplicity of notation, is denoted  by $\|\cdot\|$. As a result, for any matrix $A\in \R^{n\times n}$ and any  $\x\in \R^{n \times d}$, we have $\|A\x\|\leq \|A\|\|\x\|$. Note that for any $\x\in \R^{n \times d}$ and $u=\x v$ for some $v\in \R^{d\times 1}$ with $\|v\|=1$, by the Cauchy-Schwartz inequality, we have that
\[|u_i|=|x_i v|\leq \|x_i\|\|v\|=\|x_i\|.\]
for all $i\in[n]$. Therefore, 
\[\|u\|=\|\x v\|\leq \sqrt{n}\max_{i\in[n]}\|x_i\|,\]
 and hence, 
 \begin{align}\label{eqn:matrixnorm}
     \|\x\|:=\sup_{v\in\R^{d\times 1}:\|v\|=1}\|\x v\|\leq \sqrt{n}\max_{i\in[n]}\|x_i\|.
 \end{align}
}

\myclearpage

\section{Problem Formulation}\label{section:problem-formulation}
Consider a network of $n \in \integers^+$ agents whose communication topology is given by a sequence of time-varying directed graphs $ \{\G(t)\}_{t\geq 0}$; here $ \G(t)=([n],\mathcal{E}(t))$, where $ \mathcal{E}(t) \subseteq [n]\times [n]$ is the set of edges at time $t\geq 0$. Throughout this paper, we assume that agent $ i\in [n]$, at time $ t $, can obtain information from its out-neighbors at that time, i.e., the set of agents in $ \{j \in [n] \ | \ (i,j) \in \mathcal{E}(t) \}$. Suppose now that each agent $i\in [n]$ is equipped with a differentiable convex function $ \map{f_i}{\real^d}{\real} $, only available to this agent. The main objective is to provide a continuous-time dynamics, distributed at each time over the network, that converges to an optimizer of
\begin{align}\label{eqn:mainproblem}
  \mbox{minimize }& F(x)=\sum_{i=1}^nf_i(x).
\end{align}
{In other words, we are seeking to find a minimizer of  \eqref{eqn:mainproblem} at each node using a continuous-time dynamics where each node's update rule only depends on the information shared by the neighbors' over the underlying time-varying graph. The main idea of the many existing distributed optimization algorithms is to drive the nodes' states to the consensus subspace while utilizing a form of gradient flow dynamics at each node.  In this work, we formalize this intuition mathematically and introduce general conditions that allow us to conclude that any dynamics satisfying our general framework results in convergence to an optimal point of \eqref{eqn:mainproblem} at each node. We then show that the convergence of many existing algorithms is deduced from this result.}

We make the following assumption throughout the paper. 
\begin{assumption}[Assumption on the Objective Function]\label{assum:optimization}
We assume that $\map{f_i}{\real^d}{\real}$ is convex and differentiable with bounded gradients for all $i\in [n]$ and \begin{align}\label{eqn:nonemptyset}
X^*=\argmin_{x\in\real^d } F(x)
\end{align}
is nonempty. 
\end{assumption}

We also 
use the notations 
\[\fLift(\x):=(f_1(x_1),f_2(x_2),\ldots,f_n(x_n))^T,\] 
and 
\[\gradient \fLift(\x):=\begin{bmatrix}\gradient f_1(x_1)\\ \gradient f_2(x_2)\\\vdots\\ \gradient f_n(x_n)\end{bmatrix}\in \R^{nd},\] 
where $\x\in\R^{n\times d}$.

\myclearpage

\section{A Unified Approach to Distributed Optimization}\label{section:system-approach}
In this section, we formally present our general model and the observer-based approach to distributed optimization. As we will demonstrate later on, many of the previous works in distributed optimization are implementations of this scheme. 

We start with a general framework for dynamics which describes the evolution of individual agent's states depending on both their \emph{observations of the average state} and the \emph{(partial) states of their neighboring agents}. In particular, we consider the case where the agents estimate the average state through a dynamic observer. In its most general form, we consider the dynamics
\begin{align}\label{eqn:inputoutput}
	\dot{x}_i(t)&= p_i(t,\x(t),\w(t),u_i(t)),\cr
	\dot{w}_i(t)&=q_i(t,\x(t), \w(t)),\cr
	{y}_i(t)&=h_i(t,x_i(t),w_i(t),u_i(t)),
\end{align}
where 
\begin{enumerate}[a.]
		\item 	$\x(t)=\begin{bmatrix} x_1(t)\\x_2(t)\\\vdots\\x_n(t)\end{bmatrix}\in\real^{n\times d} $ and $ x_i(t) \in \real^{1\times d} $ denotes the internal state of agent $ i $ that is shared with its neighboring agents at time $t\geq 0$,
		\item $w_i(t) \in \real^{n_i}$ and $y_i(t) \in \real^d $,  for some $n_i\geq 0$, and
		\item $u_i(t)\in\mathbb{R}^d$ is the control input on the dynamics of agent $i\in [n]$.
\end{enumerate}	
Here, $ \map{p_i}{\real\times \real^{n\times d}\times \real^{N}\times \real^d}{\real^d} $, for $N=\sum_{i=1}^nn_i$, and $ \map{q_i}{\real\times \real^{n\times d}\times \real^{N}}{\real^{n_i}} $ are mappings with sufficient continuity properties such that there exists a continuous solution to this dynamical system. {For example, if we let $\hat{p}_i(t,\x,\w)=p_i(t,\x,\w,u_i(t))$, and assume that 
for all $i\in[n]$, $\hat{p}_i$ and $q_i$ are globally Lipschitz in the last two arguments and piece-wise continuous in the first, then the solution to \eqref{eqn:inputoutput} exists over  $[t_0,t_1]$ for any $t_0\leq t_1$ and is differentiable (see e.g.\ Theorem~3.2 in~\cite{khalil1996nonlinear})}. Finally $\map{h_i}{\real\times \real^{d}\times \real^{n_i}}\real^{d}$. It is worth pointing out that extensions of this formulation to differential inclusions is possible, but we avoid this for simplicity of presentation. The next two definitions are central to our study. 

We often say that the dynamics \eqref{eqn:inputoutput} is distributed with respect to an underlying time-varying graph process $\{([n],\E(t))\}_{t\geq 0}$ if 
 \[
 \frac{\partial p_i}{\partial x_j}=0, \quad \frac{\partial p_i}{\partial w_j}=0, \quad \frac{\partial q_i}{\partial x_j}=0, \ \mathrm{and} \ \frac{\partial q_i}{\partial w_j}=0,
 \]
for all {$(i,j)\not\in \E(t)$} and all $t\geq 0$, i.e., the update rule of agent $i$ at time $t$ only depends on information coming from its (out-)neighbors.

We next introduce the main focus of this work, which is the study of distributed flow tracker dynamics. 

\begin{definition}\label{def:flow-tracker}
	For a given graph sequence $\{\G(t)\}_{t\geq 0}$, we say that the system \eqref{eqn:inputoutput} is \emph{distributed flow tracker} {with respect to a set $Q\subseteq \R^{n\times d}\times \R^N$} if it satisfies the following properties
	\begin{enumerate}[I.]
		\item \textit{Distributed:} It is distributed with respect to $\{\G(t)\}_{t\geq 0}$.
		\item \textit{Average-input tracker}: For some constant $c_1\in \real^+$, we have
		\begin{align}\label{eqn:inputtrack}
		\dot{\bar{x}}(t)= c_1\sum_{i=1}^nu_i(t).
		\end{align}
		\item \textit{Average-state observer:}  {For all $i\in[n]$ 
		\begin{align}\label{eqn:observer}
		\|y_i(t)-\bar{x}(t)\|\leq c_2\left(\lambda^t\|\x(0)\|+\int_{0}^t\lambda^{t-s}\|\u(s)\|ds\right)
		\end{align}
		holds for all initial condition $(\x(0),\w(0))\in Q\subseteq \R^{n\times d}\times \R^N$ and all smooth control inputs $\u:\R^+\to\R^{n\times d}$}, for some  $\lambda\in[0,1)$ and $c_2\in \real^+$ (possibly dependent on $\u$). 
	\end{enumerate}
\end{definition}
{Note that whether a dynamics is a distributed flow tracker or not  depends heavily on the structure of the algorithm, i.e., the vector field defined by $p_i,q_i$, the observer $y_i(t)$ for $i\in[n]$ in \eqref{eqn:inputoutput}, and the network connectivity over time. In Section~\ref{section:implications} we discuss different algorithms 
and different connectivity conditions that result in a distributed flow tracker dynamics.}
{We note that given the local nature of available information about the objective functions, many protocols that solve~\eqref{eqn:mainproblem} rely on some type of averaging dynamics among the agents.  In fact, as we will show later, all the mentioned averaging schemes ensure the \textit{average-state observer} condition above.}

Note that due to equivalence of norms in finite dimensional spaces,  the underlying norm in \eqref{eqn:observer} does not play a role. 
Before stating the main result, we make the following standard assumption on the class of time-varying step-sizes that will be used in the subsequent results. 

\begin{assumption}[Assumption on the Step-size]\label{assum:stepsize}
We assume that $\alpha:[0,\infty)\to (0,1]$ is a non-increasing function, $\int_0^{\infty} \alpha(t)dt=\infty$, and $\int_0^{\infty} \alpha^2(t)dt<\infty$.
\end{assumption}

Our main result, which will be proved in Section~\ref{section:proofs}, is the following. 
\begin{theorem}\label{thrm:MAIN}
{Consider a dynamics~\eqref{eqn:inputoutput} that is a distributed flow tracker dynamics with respect to a set $Q$ for a given graph sequence $\{\G(t)\}_{t\geq 0}$. Then, for any initial condition $(\x(0),\w(0))\in Q$, the feedback law
\[u_i(t)= -\alpha(t)\gradient f_i(y_i(t))\] 
for~\eqref{eqn:inputoutput} solves the distributed optimization problem \eqref{eqn:mainproblem} satisfying Assumption \ref{assum:optimization}, where $\alpha(\cdot)$ satisfies the step-size Assumption \ref{assum:stepsize}, i.e., for any solution of \eqref{eqn:inputoutput} starting at $(\x(0),\w(0))\in Q$, there exists an $x^*\in X^*$ such that $\lim_{t\to\infty}y_i(t)=x^*$ for all $i\in [n]$}.
\end{theorem}

{Note that the assumption that $\alpha(t)$ is not integrable but it is square integrable (Assumption~\ref{assum:stepsize}) is a natural assumption in the discrete-time setting and it arises in several distributed optimization and averaging results in discrete time (see e.g., \cite{AN-AO-PAP:10,sundhar2012new,
Nedic2013,aghajan2022distributed,reisizadeh2022distributed}). Due to the nature of gradient flow dynamics in continuous-time, which does not require diminishing step-size for convergence, one may wonder if the above result can be generalized for non-diminishing and constant step-size $\alpha(t)=\alpha$ for all $t$ and some $\alpha>0$. However, as we will show later (after Theorem~\ref{thrm:DO-Averaging}), this condition cannot be relaxed for the general class of distributed tracker dynamics.}

\myclearpage

\section{Implications}\label{section:implications}
In this section, we discuss the implications of the main result on a number of the existing dynamics for solving the distributed optimization problem. In particular, we show that they are instances of the distributed flow tracker dynamics and their convergence are implied by Theorem~\ref{thrm:MAIN}.

Let us start by discussing some preliminary definitions on distributed averaging dynamics which play a central role in ensuring average observer property~\eqref{eqn:observer}. Let $\{L(t)\}_{t\geq 0}$ be a sequence of matrices, where either $L(t)$, and/or $L^T(t)$, is a generalized Laplacian matrix for all $t\geq 0$. Additionally, we assume that the sequence $ \{L(t)\}_{t\geq 0}$ is measurable and locally essentially bounded, and hence the solution $\Phi(t,s)$, in the Carath\'eodory sense, to the system of ordinary differential equations 
\begin{equation}\label{eqn:LtoPhi}
\dot{\Phi}(t,s)=-L(t)\Phi(t,s),
\end{equation}
with the initial condition $\Phi(s,s)=I$, is well-defined for any $t\geq s\geq 0$, see~\cite{JC:08-csm-yo}. Clearly, $\Phi(t,s)$ is the transition matrix associated with the distributed averaging dynamics on $ \real^n $ given by
\begin{align*}
\dot{x}(t)=-L(t)x(t),
\end{align*} 
with some initial condition $ x(0)=x_0\in \real^n $.  It is important to note that in the case where $L(t)$ (respectively, $L^T(t)$)  is a generalized Laplacian matrix for all $t\geq 0$, $\Phi(t,s)$ is a column-stochastic (row-stochastic) matrix for all $t\geq s\geq 0$.

As we will discuss later, to ensure the exponential rate \eqref{eqn:observer}, it is desirable that the solutions to~\eqref{eqn:LtoPhi} satisfy some additional properties, which we outline next. 
\begin{definition}\label{def:flow}
	We say that $\{\Phi(t,s)\}$ is a \emph{stochastic flow} if $\Phi(t,s)\in \S$ for all $t\geq s\geq 0$. We say that the flow  $\{\Phi(t,s)\}$ is \emph{weakly ergodic} if for any $s\geq 0$, $\Phi(t,s)$ converges to $\S_1$, i.e.\ $\lim_{t\to\infty}d(\Phi(t,s),\S_1)=0$. We say that a flow is \emph{in class $P^*$} if $\Phi(t,s)\one\geq p^*\one$ for some $\ps>0$. Finally, we say that a flow $\{\Phi(t,s)\}$ is \emph{weakly exponentially ergodic} if $d(\Phi(t,s),\S_1)\leq a\lambda^{t-s}$ for some $\lambda\in (0,1)$ and $a>0$, and all $t\geq s\geq 0$.
\end{definition}

\subsection{Distributed Optimization using Averaging}
Let us first consider a continuous version of the distributed optimization through averaging, which was introduced in \cite{AN-AO:09} in discrete-time setting. The continuous-time version is given by
\begin{align*}
\dot{\x}(t)&= -L(t)\x(t)-\alpha(t)\gradient\fLift(\x(t)),
\end{align*}
where $L(t)$ is a ``sufficiently mixing'' weight-balanced matrix, as will be discussed shortly. Recall that $L(t)$ is weight-balanced if both $L(t)$ and $L^T(t)$ are generalized Laplacian matrices. 
Note that the above dynamics can be written as 
\begin{align}\label{eqn:distopt-averaging}
\dot{\x}(t)&= -L(t)\x(t)+\u(t),\cr
\y(t)&=\x(t),
\end{align}
with $\u(t)=-\alpha(t)\gradient \fLift(\y(t))$, {which is an instance of dynamics~\eqref{eqn:inputoutput} with $N=0$.} In fact, with enough mixing, we show that this dynamics is a distributed flow-tracker dynamics with respect to $\R^{n\times d}$ and, consequently, we have the following result, which is proved in the Appendix.  
\begin{theorem}\label{thrm:DO-Averaging}
  Suppose that $L(t)$ is a weight-balanced matrix for all $t\geq 0$ such that the resulting flow $\Phi(t,s)$, defined by \eqref{eqn:PhitPhis}, is weakly exponentially ergodic. Then for any $\alpha(\cdot)$ satisfying the step-size Assumption \eqref{assum:stepsize}, the dynamics \eqref{eqn:distopt-averaging} is a distributed flow tracker dynamics {with respect to $Q=\R^{n\times d}$} and hence, $\lim_{t\to\infty}x_i(t)=x^*$ for some $x^*\in X^*$ for any distributed optimization problem \eqref{eqn:mainproblem} satisfying Assumption~\ref{assum:optimization}.
\end{theorem}

{Similar to discrete-time variation~\eqref{eqn:distopt-averaging}, as shown in~\cite{qu2017harnessing}, we show that we cannot have  convergence to an optimal point without diminishing step-size $\alpha(t)$ in~\eqref{eqn:distopt-averaging}. For this, consider the simple scenario with two agents on a connected time-invarying undirected graph and $d=1$ with 
\[L(t)=\begin{bmatrix}1&-1\\-1&1
\end{bmatrix}\]
and cost functions 
\[
f_1(x)=\frac{1}{2}(x-1)^2 \quad \mathrm{and}  \quad f_2(x)=\frac{1}{2}(x+1)^2.
\]
First, note that the minimizer of $ (f_1+f_2) $ is at $ x=0 $. 
For a constant rate $ \alpha >0 $, consider the above averaging-based distributed optimization dynamics, which can be written as 
\begin{align*}
    \dot{x}(t)=    \underbrace{\begin{bmatrix}
        -1-\alpha & 1\\
        1 & -1-\alpha
    \end{bmatrix}}_{A}x(t)
    +    \alpha\begin{bmatrix}
        1\\-1
    \end{bmatrix}.
\end{align*}
The solution then is given by 
\[
x(t)=e^{At}x(0)+\int_0^te^{A(t-\tau)}\begin{bmatrix}
    1\\-1
\end{bmatrix}\alpha d\tau.
\]
Note that $ A $ is negative-definite with eigenvalues $\lambda_1=-\alpha$ and $\lambda_2=-2-\alpha$. In particular, 
the vector $(1,-1)^T$ is an eigenvector of $ A $ with the corresponding eigenvalue of $(-2-\alpha)$. Therefore,
\[
x(t)=e^{At}x(0)+\frac{\alpha}{2+\alpha}\begin{bmatrix}
    1\\-1
\end{bmatrix}.
\]

As $ t\rightarrow \infty $, the trajectory convergences to $\frac{\alpha}{2+\alpha}[1,-1]^T$, which is not equal to $x^*=0$, the minimizer of $f_1+f_2$.}

\subsection{Distributed Optimization using Push-Sum}
We now introduce a continuous-time variation of the (discrete-time) push-sum based optimization algorithm, studied for time-invariant scenarios in~\cite{Tsianos1,Tsianos2,TsianosThesis} and later extended to time-varying graphs in~\cite{Nedic2013}. The continuous-time version is given by
 \begin{align*}
 \dot{\x}(t)&= -L(t)\x(t) - \alpha(t) \gradient\fLift(\y(t)),\cr
 \dot{w}(t)&=-L(t)w(t),\cr
 y_i(t)&=\frac{x_i(t)}{w_i(t)},
 \end{align*}
where $w_i(0)=1$ for all $i\in [n]$. As in the discrete-time case, the advantage of this dynamics  to \eqref{eqn:distopt-averaging} is that $L(t)$ need not be weight-balanced.

Again, we can view this dynamics as the following input-output dynamics \begin{align}\label{eqn:DO-push}
 \dot{\x}(t)&= -L(t)\x(t)+\u(t),\cr
 \dot{w}(t)&=-L(t)w(t),\cr
 y_i(t)&=\frac{x_i(t)}{w_i(t)},
 \end{align}
 with $N=n$ and the feedback $\u(t)=-\alpha(t) \gradient \fLift(\y(t))$ and similar to Theorem~\ref{thrm:DO-Averaging}, if the sequence $\{L(t)\}$ is ``sufficiently mixing'', this dynamics is a flow-tracker dynamics.

\begin{theorem}\label{thrm:DO-push}
Consider the dynamics~\eqref{eqn:DO-push} and suppose that $\{\Phi(t,s)\}$ defined by \eqref{eqn:PhitPhis} is a class $P^*$ weakly exponentially ergodic flow. Then, the dynamics \eqref{eqn:DO-push} is a flow tracker dynamics {with respect to $Q=\R^{n\times d}\times \{\ones\}$}.  Moreover, for the feedback $\u(t)=-\alpha(t) \gradient\fLift(\y(t))$ with $ \alpha(\cdot) $ satisfying Assumption~\ref{assum:stepsize}, we have $\lim_{t\to\infty}y_i(t)=x^*$ for some $x^*\in X^*$, for all $i\in [n]$, for a distributed optimization problem \eqref{eqn:mainproblem} satisfying Assumption~\eqref{assum:optimization}. \end{theorem}

\subsection{Distributed Optimization using Saddle-Point Dynamics}
Another approach to solve the distributed optimization problem \eqref{eqn:mainproblem} is through saddle-point like dynamics, originally established in~\cite{JW-NE:10,BG-JC:14-tac}; this dynamics is given by
\begin{align*}
	\dot{\x}(t)&=-aL(t)\x(t)-L(t)\w(t)-\alpha(t)\nabla \tilde{F}(\x(t))\cr
	\dot{\w}(t)&=L(t)\x(t),
\end{align*}
where $ \x(t),\w(t) \in \real^{n\times d} $ for all $ t \geq 0 $ and some $a>0$. 

This dynamics can be viewed as the following input-output dynamics
\begin{align}\label{eqn:DO-SP}
	\dot{\x}(t)&=-aL(t)\x(t)-L(t)\w(t)+\u(t)\cr
	\dot{\w}(t)&=L(t)\x(t)\cr 
	\y(t)&=\x(t),
\end{align}
with the feedback $\u(t)=-\alpha(t)\nabla \tilde{F}(\x(t))$ and $N=nd$.

We show that for sufficiently mixing $\{L(t)\}$, this dynamics is a flow tracker and again Theorem~\ref{thrm:MAIN} applies here.

\begin{theorem}\label{thrm:DO-SP}
	Let $\{L(t)\}$ be a sequence of weight-balanced Laplacian matrices such that
	\begin{align}\label{eq:minc-beta}
	\int_{t}^{t+T}L_{minc}(\tau)d\tau\geq \beta,
	\end{align}
	for some $\beta>0$, some $T>0$, and all $t\geq 0$, where 
	\[L_{minc}(\tau)=\min_{\emptyset\not= S\subset[m]}L_{S\bar{S}}(\tau),\]
	is the minimum-cut at time $\tau$. Then, the saddle-point dynamics \eqref{eqn:DO-SP} is a flow-tracker dynamics {with respect to~$Q=\R^{n\times d}\times \R^{n\times d}$}  for $a\geq 5$. As a result, for a distributed optimization problem \eqref{eqn:mainproblem} satisfying Assumption \eqref{assum:optimization}, with the feedback $\u(t)\in -\alpha(t)\nabla  \bar{F}(\x(t))$,  where the step-size  $\alpha(t)$ satisfies Assumption~\ref{assum:stepsize}, for all initial conditions $\x(0),\w(0)\in\R^{n\times d}$ we have $\lim_{t\to\infty}x_i(t)=x^*$ for all $i\in[n]$ and some $x^*\in X^*$.  
\end{theorem}

\subsection{Modified Saddle-Point Dynamics with Push-sum}

In the recent work~\cite{BT-BG:19-tac}, we considered the following continuous-time dynamics for solving the distributed optimization problem, where for simplicity, we have assumed that the  the state of each agent is a scalar
 \begin{align}\label{eqn:saddlepointpush}
  \dot{\x}(t)&= -a L(t)\x(t)-L(t)\z(t)+\u(t) ,\cr
  \dot{\z}(t)&=L(t)\x(t)\cr
  \dot{v}(t)&=-L(t)v(t),\cr
  y_i(t)&=\frac{x_i(t)}{v_i(t)},
\end{align}
where $ t \geq 0 $, $ \x(t)=(x_1(t),\ldots, x_n(t))^T\in \real^{n} $, $ \z(t)=(z_1(t),\ldots, z_n(t))^T\in \real^{n} $,  $ x_i(t) ,z_i(t) \in \real $ and $v_i(t)\in \R$ are the states of  $i$th agent,
$y_i(t)$ is the agent $i$'s estimate of the solution to~\eqref{eqn:mainproblem}, and 
$\u(t)\in\real^{n} $ is sufficiently well-behaved. We assume here that the entries of $ L(t) $ are uniformly bounded over time. Note that if we let $w(t):=(z^T(t),v^T(t))^T$, then the dynamics \eqref{eqn:saddlepointpush} can be viewed as an instance of the distributed input-output dynamics \eqref{eqn:inputoutput}. The next result is a restatement of~\cite[Proposition~3.3]{BT-BG:19-tac}. 
\begin{proposition}\label{prop:tracking}
Consider the dynamic~\eqref{eqn:inputoutput}, for an arbitrary $ (\x(0),\z(0)) \in \real^{n}\times \real^{n} $ and let $v_i(0)=1$ for all $i\in[n]$. Suppose that the sequence of Laplacian matrices $\{L(t)\}$ admits a common stationary distribution $\pi>0$  and has a minimum cut $\gamma>0$.
Then there exists a time $t_0\geq 0$ such that for some $c_2\in \R$, any $i\in [n]$, and $t\geq t_0$ we have that
\begin{align}\label{eqn:observer-SPP}
		\|y_i(t)-\bar{\x}(t)\|\leq c_2(\lambda^t\|\x(0)\|+\int_{0}^t\lambda^{t-\tau}\|\u(\tau)\|d\tau),
\end{align}
where $a\geq 5$, $\lambda=e^{-\frac{2a\pi_{\min}\gamma}{n^2}}\in (0,1)$, and $ \bar{\x}(t)=\frac{1}{n} \sum_{i=1}^nx_i(t) $. 
\end{proposition}

Note that the inequality~\eqref{eqn:observer-SPP} along with $\dot{\bar{\x}}(t)=\bar{\u}(t)$ are  the exact requirements for generating a flow-tracker according to Definition~\ref{def:flow-tracker}. Therefore, the main result of the mentioned paper~\cite[Thoerem~2.5]{BT-BG:19-tac} follows as a corollary of Theorem~\ref{thrm:MAIN}.

\begin{theorem}\label{thrm:mainresult-SPP}
 Suppose that the sequence of Laplacian matrices $\{L(t)\}$ admits a common stationary distribution $\pi>0$  and has a minimum cut $\gamma>0$. {Then the dynamics~\eqref{eqn:saddlepointpush} is a distributed flow tracker dynamics with respect to~$Q=R^n\times (\R^n\times \{\ones\})$. As a result if Assumptions~\ref{assum:optimization} and~\ref{assum:stepsize} hold, for any initial conditions~$\x(0),\z(0)\in\R^n$ and $v(0)=\ones$, we have $\lim_{t\to\infty}y_i(t)=x^*$, for all $i\in [n]$ and some $x^*\in X^*$, for the solutions of \eqref{assum:optimization} with $\u(t)=-\alpha(t)\nabla \tilde{F}(\x(t))$.}
\end{theorem} 

We finish this section with a remark. One of the objectives of this work is to systematically decouple the role of the mixing of information and distributed optimization. In the above results, mixing of information is ensured by assuming conditions on exponential ergodicity of a generalized Laplacian process $\{L(t)\}$.   While our focus in this work is not on developing sharpest conditions on the information exchange mechanisms that ensure  exponentially ergodic flows, there are many results available in the literature that serve this purpose, including~\cite{Touri2014,Hendrickx2013,Bolouki2014,martin2013continuous}, that could simply be inserted here to obtain convergence results for distributed optimization dynamics.

\myclearpage

\section{Technical Details and Proofs}\label{section:proofs}

In this section, we provide the proofs of our results.
\subsection{Proof of Theorem \ref{thrm:MAIN}}
We start by proving the main result. 
\begin{proof}
Let $x^*\in X^*$ and consider the Lyapunov candidate $V:\R^d\to\R$ given by
	\begin{align}\label{eqn:V}
	V(x)=\frac{1}{2}\|x-x^*\|^2.
	\end{align}
	
    The function $V$ is smooth.  {Then, consider a solution of a distributed flow tracker dynamics~\eqref{eqn:inputoutput} with respect to $Q$, started at $(\x(0),\w(0))\in Q$}. We examine the  derivative of this function along such a trajectory of the average dynamics which satisfies \eqref{eqn:inputtrack}. 
	We have that
	\begin{align}\label{eqn:Vdot}
	\dot{V}(\bar{x}(t))=-\alpha(t)\sum_{i=1}^n \nabla f_i(y_i(t))(\bar{x}(t)-x^*),
	\end{align} 
	where by proper scaling of the step-size sequence $\alpha(t)$, we assume that $c_1=1$ in \eqref{eqn:inputtrack}.	
	Next, we have
	\begin{align}\label{eqn:Vintermediate}  
	\dot{V}(\bar{x}(t)) =&-\alpha(t)\sum_{i=1}^n\nabla f_i(y_i(t))(\bar{x}(t)-y_i(t))+\alpha(t)\sum_{i=1}^n\nabla f_i(y_i(t))(x^*-y_i(t))\\
	\leq& \alpha(t)\Lip\sum_{i=1}^n\|\bar{x}(t)-y_i(t)\|-\alpha(t)(\sum_{i=1}^nf_i(y_i(t))-F(x^*)),
	\end{align}
	where the first term in the above inequality follows from the fact that the gradient of $f_i$s are bounded by $\Lip$ and application of Cauchy-Schwartz inequality,  and the second inequality follows from convexity of $f_i$'s as: \[\nabla f_i(y_i(t))(x^*-y_i(t))\leq f_i(x^*)-f_i(y_i(t)).\]
	By the bounded subgradient property of $f_i$s, we have \[\|f_i(y_i(t))-f_i(\bar{x})\|\leq \Lip\|y_i(t)-\bar{x}(t)\|.\] Using this in \eqref{eqn:Vintermediate}, we conclude that
	\begin{align}\label{eqn:finalA}
	\dot{V}(\bar{x}(t)) &\leq 2\Lip\sum_{i=1}^n\alpha(t)\|\bar{x}(t)-y_i(t)\|-\alpha(t)(F(\bar{x}(t))-F(x^*)). 
	\end{align}
    Integrating both sides of the above inequality over $[0,t]$ interval for $t>0$, we have
	\begin{align}\label{eqn:finalV}
	V(\bar{x}(t))-V(\bar{x}(0))\leq 2\Lip\sum_{i=1}^n\int_{0}^t\alpha(s)\|\bar{x}(s)-y_i(s)\|ds\cr
	-\int_{0}^t\alpha(s)(F(\bar{x}(s))-F(x^*))ds.
	\end{align}
	We next show that $V(\bar{x}(t))$ converges. For convenience, let 
	\begin{align*}
	h(t)&=2\Lip\sum_{i=1}^n\int_{0}^t\alpha(s)\|\bar{x}(s)-y_i(s)\|ds.
	\end{align*}
	Then, using the average-state observer property \eqref{eqn:observer}, we have{ 
	\begin{align*}
	h(t)&\leq 2\Lip c_2\sum_{i=1}^n\int_{0}^t\alpha(s)
	    \left(\lambda^s\|x(0)\|+\int_0^s\lambda^{s-\tau}\|\u(\tau)\|d\tau\right) ds\cr 
	&\leq  2\Lip c_2\sum_{i=1}^n\int_{0}^t\alpha(s)
	    \left(\lambda^s\|x(0)\|+\sqrt{n}\Lip\int_0^s\lambda^{s-\tau}\alpha(\tau)d\tau\right) ds,
	\end{align*}
	where the last inequality follows from \eqref{eqn:matrixnorm} and the fact that \[\|\u(s)\|=\sqrt{n}\max_{i\in [n]}\|\alpha(s)\nabla f_i(y_i(s))\|\leq \sqrt{n}\Lip\alpha(s).\]}
	But 
	\[\int_{0}^t\alpha(s)
	    \lambda^s\|x(0)\|ds\leq \alpha(0)\|x(0)\|\int_{0}^t
	    \lambda^sds\leq \frac{\alpha(0)\|x(0)\|}{1-\lambda},\]
	as $\alpha(s)$ is a non-increasing function and $\lambda\in (0,1)$. Also, by Lemma~\ref{lemma:aux}, we have: 
	\begin{align*}
	    \int_{0}^t\alpha(s)\int_0^s\lambda^{s-\tau}\alpha(\tau)d\tau ds\leq \frac{1-\lambda}{\log \lambda }\int_{0}^\infty \alpha^2(s)ds. 
	\end{align*}
	
	Combining the above two observations, we conclude that $h(t)$ is uniformly bounded, i.e., there exists some $\gamma>0$ such that $h(t)\leq \gamma$ for all $t\geq 0$. Note that $h(t)$ is a non-decreasing function and hence,  $\lim_{t\to\infty}h(t)$ exists and $\lim_{t\to\infty}h(t)\leq \gamma$. By \eqref{eqn:finalV}, and the fact that $F(x)\geq F(x^*)$, for any $x\in\R^d$, we have that
	\begin{align*}
	V(\bar{x}(t_2))-V(\bar{x}(t_1))\leq h(t_2)-h(t_1),
	\end{align*}
	for any $0\leq t_1<t_2$. This implies that $V(\bar{x}(t))$ and hence, $\bar{x}(t)$ is bounded and
	\begin{align*}
	\limsup_{t\to\infty}V(\bar{x}(t))-\liminf_{t\to\infty}V(\bar{x}(t)) \leq \limsup_{t\to\infty}h(t)-\liminf_{t\to\infty}h(t)=0,
	\end{align*}
	and hence, $\lim_{t\to\infty}V(\bar{x}(t))$ 
	exists.
	
	Also, since $V(\bar{x}(t))\geq 0$ and $F(\bar{x}(t))\geq F(x^*)$, using~\eqref{eqn:finalV} we conclude that
	\begin{align}\label{eqn:finalV-2}
	0\leq \int_{0}^\infty\alpha(s)(F(\bar{x}(s))-F(x^*))ds \leq \gamma-V(\bar{x}(0))<\infty,
	\end{align}
	which implies that $\liminf_{t\to\infty}F(\bar{x}(t))=F(x^*)$. In other words, there exists a subsequence $x_{t_k}=x(t_k)$ of $\{x(t)\}$ that converges to a point $\hat{x}\in X^*$ (note that $\bar{x}(t)$ is bounded). Since the above arguments hold for any $x^*\in X$, we may repeat the argument for $\hat{x}$ and conclude that, in this case, convergence of $V(\bar{x}(t))$ implies that $\lim_{t\to\infty}\bar{x}(t)=\hat{x}$. Finally, note that $\lim_{t\to\infty} \|y_i(t)-\bar{x}(t)\|=0$ for any $i\in [m]$ and hence, $\lim_{t\to\infty}y_i(t)=\hat{x}$, concluding the result. 	
\end{proof}
\subsection{Proof Theorem~\ref{thrm:DO-Averaging}}
We now move on to our next proof. 
\begin{proof}
Since $L(t)$ is a weight-balanced matrix for any $t\geq 0$, the dynamics~\eqref{eqn:distopt-averaging} satisfies the average input tracking property \eqref{eqn:inputtrack}. To show that~\eqref{eqn:inputoutput} holds, we have
\begin{align}\label{eqn:DO-AV-twodyn}
\x(t)&=\Phi(t,0)\x(0)+\int_{0}^t\Phi(t,s)\u(s)ds,\cr 
{\bar{x}}(t)&=\bar{x}(0)+\int_{0}^t\bar{u}(s)ds,
\end{align}
where the second equality holds due to the fact that $L(t)$ is a weight balanced matrix and hence, $\Phi(t,s)$ is a doubly-stochastic matrix for any $t\geq s\geq 0$. By subtracting the two equations, and using the triangle inequality, we get
\begin{align*}
\|x_i(t)-\bar{x}(t)\|&\leq \|\x(t)-\ones{\bar{x}}(t) \|\cr 
&\leq \|\Phi(t,0)-\frac{1}{m}\ones\ones^T\|\|\x(0)\|+\int_{0}^t\|\Phi(t,s)-\frac{1}{m}\ones\ones^T\|\|\u(s)\|ds,\cr 
&\leq c(\lambda^t\|\x(0)\|+\int_{0}^t\lambda^{t-s}\|\u(s)\|ds),
\end{align*}
for some $c>0$, where the last inequality follows from the fact that the flow $\Phi(t,s)$ is an exponentially ergodic flow.

Therefore, dynamics~\eqref{eqn:distopt-averaging} satisfying the assumptions of Theorem~\ref{thrm:DO-Averaging} is a  distributed flow tracker. As a result, if the optimization problem~\eqref{eqn:mainproblem} satisfies the assumption of Theorem~\ref{thrm:MAIN}, $\lim_{t\to\infty}x_i(t)=x^*$ for all $i\in [n]$ and some $x^*\in X^*$.
\end{proof}

\subsection{Proof of Theorem~\ref{thrm:DO-push}}

Our first result demonstrates that each output $y_i(t)$, $ i \in[n] $, is an observer for $\bar{x}(t)$.
\begin{proposition}\label{prop:observer}
	Consider the dynamics~\eqref{eqn:DO-push} with the initial condition $(\x(0),\w(0))$, where $ \w(0)=\one $ and $ \x(0)\in \real^{n\times d} $. Suppose that $\{\Phi(t,s)\}_{t\geq s\geq 0}$ is a weakly exponentially ergodic flow generated by~\eqref{eqn:LtoPhi} which is also in class $P^*$. Then
	\begin{align}\label{eqn:observer3}
	\|y_i(t)-\bar{x}(t)\|\leq \frac{1}{\ps}(\lambda^t\|\x(0)\|+3\int_{0}^t\lambda^{t-s}\|\u(s)\|ds),
	\end{align}
	where $ \lambda \in (0,1) $ and $ p^* >0 $. 
\end{proposition}
\begin{proof}
	For any $t\geq s\geq 0$, let $\pi(t,s)\one^T$ be the projection of $\Phi(t,s)$ on the set of rank-one stochastic matrices $\S_1$. Note that $\pi(t,s)$ is a stochastic vector. Let us denote the residual $R(t,s)=\Phi(t,s)-\pi(t,s)\one^T$. Because of the exponential ergodic property of the flow, it follows that $\|R(t,s)\|\leq \lambda^{t-s}$.
	Due to the semigroup property of the flow, we have
	\begin{align*}
	\Phi(t,0)=\Phi(t,s)\Phi(s,0)&=\left(\pi(t,s)\one^T+R(t,s)\right)\Phi(s,0)\cr
	&=\pi(t,s)\one^T+R(t,s)\Phi(s,0),
	\end{align*}
	where the last equality follows from the fact that $\Phi(s,0)\in \S$. As a result, we have
	\begin{align}\label{eqn:PhitPhis}
	\Phi(t,s)&=\pi(t,s)\one^T+R(t,s),\cr 
	&=\Phi(t,0)-R(t,s)(\Phi(s,0)-I).
	\end{align}
	The last equality shows that if $\Phi(t,s)$ is sufficiently close to $\S_1$ (i.e.\ $\|R(t,s)\|$ is small), then $\Phi(t,0)$ is a good approximation for $\Phi(t,s)$. This fact is central to our later development. 
	
	As a result of \eqref{eqn:PhitPhis}, we have
	\begin{align}
	\x(t)&=\Phi(t,0)\x(0)+\int_{0}^t\Phi(t,s)\u(s)ds,\cr
	&=\Phi(t,0)\left(\x(0)+\int_{0}^t\u(s)ds\right),\cr
	&+\int_{0}^tR(t,s)(I-\Phi(s,0))\u(s)ds.
	\end{align}
	Therefore,
	\begin{align}\label{eqn:approx1}
	\|x_i(t)-\Phi_i(t,0)\left(\x(0)+\int_{0}^t\u(s)ds\right)\|\leq 2\int_{0}^t\lambda^{t-s}\|\u(s)\|ds.
	\end{align}
	On the other hand, since $\one^T\Phi(t,s)=\one^T$, we have
	\begin{align}\label{eqn:approx2}
	w_i(t)\frac{1}{m}(\one^T\x(t))&= \left(\Phi_i(t,0)\one\right)\frac{1}{m}(\one^T\x(t))\nonumber \\
	&= (\Phi_i(t,0)(\frac{1}{m}\one\one^T))(\x(0)+\int_{0}^t{\u}(s)ds).
	\end{align}
	But,
	\begin{align}\label{eqn:approx3}
	\|\Phi_i(t,0)-\Phi_i(t,0)(\frac{1}{m}\one\one^T)\|
	&=\|\pi_i(t,0)\one^T+R_i(t,0)-(\pi_i(t,0)\one^T+R_i(t,s))(\frac{1}{m}\one\one^T)\|,\cr
	&=\|R_i(t,0)(I-\frac{1}{m}\one\one^T)\|\leq 2\lambda^t.
	\end{align}
	
	Combining \eqref{eqn:approx1}, \eqref{eqn:approx2}, and \eqref{eqn:approx3} and using the triangle inequality, we have
	\begin{align*}
	\|x_i(t)&-w_i(t)\frac{1}{m}(\one^Tx(t))\|\\
	&\leq
	\|x_i(t)-\Phi_i(t,0)(\x(0)+\int_{0}^t\u(s)ds))\|\\
	&+
	\|\Phi_i(t,0)(\x(0)+\int_{0}^t\u(s)ds))-w_i(t)\frac{1}{m}(\one^Tx(t))\|,\cr
	&\leq 2\int_{0}^t\lambda^{t-s}\|\u(s)\|ds\\
	&+ \|\Phi_i(t,0)-\Phi_i(t,0)(\frac{1}{m}\one\one^T)\|\|\x(0)+\int_{0}^t\u(s)ds)\|,\cr
	&\leq \lambda^t\|x(0)\|+3\int_{0}^t\lambda^{t-s}\|\u(s)\|ds,
	\end{align*}
	where the last inequality follows form the fact that $\lambda^t\leq \lambda^{t-s}$ for $s\leq t$ and $\lambda\leq 1$. Finally, we have
	\begin{align*}
	\|y_i(t)&-\frac{1}{m}(\one^Tx(t))\|,\cr
	&=\|\frac{x_i(t)}{w_i(t)}-\frac{1}{m}(\one^Tx(t))\|,\cr
	&= \|\frac{x_i(t)-w_i(t)\frac{1}{m}(\one^Tx(t))}{w_i(t)}\|,\cr
	&\leq \frac{1}{\ps}\left(\lambda^t\|x(0)\|+3\int_{0}^t\lambda^{t-s}\|\u(s)\|ds\right),
	\end{align*}
	which finishes the proof. 
\end{proof}

\myclearpage

\section{Conclusions and future work}\label{sec:conclusions}
In this paper, we have provided an observer-based controller for a class of distributed control problems, where the agents estimate the average behaviour of the system and implement a controller that depends on their estimates of the average state. When the class $P^*$ weakly exponentially ergodic flow property holds, we have provided an upper bound for the difference of the agents' estimates and the true average. We have demonstrated that many existing distributed convex optimization algorithms are subclasses of this dynamics and, hence, their convergence properties can be concluded using our proposed dynamics.

\bibliographystyle{abbrv}
\bibliography{alias,bib,Main}

\appendix\section*{}\label{sec:appendix}
The following lemma is used in the proof of one of our main results. 
\begin{lemma}\label{lemma:aux}
  Let $\alpha(t),\beta(t):\R^+\to\R^+$ be functions such that $\alpha(t)$ is non-increasing and $\langle\alpha,\beta\rangle=\int_{0}^\infty\alpha(t)\beta(t)dt<\infty$. Then,
  \[\int_{0}^\infty\alpha(t)\int_0^t\lambda^{t-s}\beta(s)dsdt\leq \frac{(1-\lambda)}{|\log \lambda|}\langle\alpha,\beta\rangle<\infty,\]
  for $\lambda\in(0,1)$.
\end{lemma}
\begin{proof}
  Note that $\int_0^t\lambda^{t-s}\beta(s)ds=\int_{0}^t\lambda^\eta\beta(t-\eta)d\eta$. Since $\alpha(t),\beta(t),\lambda^t$ are all non-negative functions, from Tonelli's theorem~\cite{WR:87} and by changing the order of the integrals, it follows that
  \begin{align}\label{eqn:lemmaintest2}
    \int_0^\infty\alpha(t)\int_{0}^t\lambda^\eta\beta(t-\eta)d\eta&=
    \int_0^\infty\int_{0}^t\lambda^\eta\alpha(t)\beta(t-\eta)d\eta dt\cr
    &=\int_0^\infty\lambda^\eta\int_{\eta}^\infty\alpha(t)\beta(t-\eta)dtd\eta.
  \end{align}
On the other hand, since $\alpha$ is non-increasing, we have
  \begin{align}\label{eqn:lemmaintest1}
    \int_0^\infty\lambda^\eta\int_{\eta}^\infty\alpha(t)\beta(t&-\eta)dtd\eta\cr
    &\leq \int_0^\infty\lambda^\eta\int_{\eta}^\infty\alpha(t-\eta)\beta(t-\eta)dtd\eta\cr
    &=\frac{1-\lambda}{|\log \lambda|} \langle\alpha,\beta\rangle.
  \end{align}
  Combining \eqref{eqn:lemmaintest1} and \eqref{eqn:lemmaintest2} gives the desired result.
\end{proof}

\end{document}